\documentclass[11pt]{amsart}

\usepackage{amssymb}
\usepackage{hyperref} 
\theoremstyle{definition}

\theoremstyle{remark}

\numberwithin{equation}{section}
\newtheorem{result}{Result}

\newcommand{\Hom}{\operatorname{Hom}}

\newtheoremstyle{fancy}{}{}{\itshape}{}{\textsc\bgroup}{.\egroup}{ }{}
\newtheoremstyle{fanci}{}{}{\rm}{}{\textsc\bgroup}{.\egroup}{ }{}

\newtheoremstyle{cobra}{}{}{\itshape}{}{\bf\bgroup}{.\egroup}{ }{}
\theoremstyle{fancy}
\newcounter{intro}

\numberwithin{equation}{section} \swapnumbers
\newtheorem{cor}[equation]{Corollary}
\newtheorem{lem}[equation]{Lemma}
\newtheorem{prop}[equation]{Proposition}
\newtheorem{thm}[equation]{Theorem}

\newtheorem{rem}[equation]{Remark}

\newtheorem{sy}[equation]{Minkowski's Theorem}

\theoremstyle{fanci}
\newtheorem{dfn}[equation]{Definition}

\newtheorem{rems}[equation]{Remarks}

\newtheorem{nota}[equation]{Notation}
\newtheorem{notarem}[equation]{Notation and Remarks}
\newtheorem{invar}[equation]{Invariants of bi-invariant metrics}
\newtheorem{rep}[equation]{Review of representation theory}

\newcommand{\mg}{\mathfrak{g}}
\newcommand{\mh}{\mathfrak{h}}
\newcommand{\mk}{\germ{K}}

\renewcommand{\mp}{\mathfrak{p}}
\renewcommand{\mh}{\mathfrak{h}}
\newcommand{\mq}{\mathfrak{q}}
\newcommand{\bs}{\backslash}

\newcommand{\Add}{\text{Ad}}
\newcommand{\Adm}{\rm{Adm}}

\newcommand{\mb}{{\mathcal{B}}}

\newcommand{\mbsemi}{{\mathcal{B}_{\rm{semi}}}}

\newcommand{\mc}{{\mathcal{M}}}

\newcommand{\Z}{\mathbb{Z}}
\newcommand{\R}{\mathbb{R}}

\newcommand{\tg}{\tilde{g}}
\newcommand{\hh}{\tilde{h}}
\newcommand{\mH}{\mathcal{H}}

\newcommand{\germ}{\mathfrak}

\newcommand{\Ad}{\text{Ad}}

\newcommand{\SU}{\operatorname{SU}}
\newcommand{\Spec}{\operatorname{Spec}}
\newcommand{\Ind}{\operatorname{Ind}}
\newcommand{\kb}{\overline{K}}
\newcommand{\gb}{\overline{g}}
\newcommand{\qb}{\overline{h}}
\newcommand{\Lc}{{\mathcal{L}}}
\def\omi{{\overline{M_i}}}
\def\omn{{\overline{M_0}}}
\def\E{{\mathcal{E}}}
\def\ogi{{\overline{g_i}}}
\def\ogn{{\overline{g_0}}}

\newcommand{\Div}{\operatorname{Div}}
\newcommand{\Grad}{\operatorname{grad}}

\newcommand{\Symm}{\operatorname{Symm}}

\newcommand{\syst}{\operatorname{syst}}

\newcommand{\vol}{\operatorname{vol}}
\def\M{{\mathcal{M}}}

\begin{document}

\newcommand{\spacing}[1]{\renewcommand{\baselinestretch}{#1}\large\normalsize}
\spacing{1.14}

\title[Spectral Isolation of Naturally Reductive Metrics]{Spectral Isolation of
Naturally Reductive Metrics on Simple Lie Groups}

%%%%%%%%%%%%%%%%%%%%%%%%%%%%%%%%%
%%%%%%%%%%%%%% CAROLYN'S INFO  %%%%%%%%%%%
%%%%%%%%%%%%%%%%%%%%%%%%%%%%%%%%%

\author[C.S. Gordon]{Carolyn S. Gordon $^\dagger$}
\address{Dartmouth College\\ Department of Mathematics \\ Hanover, NH 03755}
\email{carolyn.s.gordon@dartmouth.edu}
\thanks{$^\dagger$Research partially supported by National Science Foundation
grant DMS 0605247}

%%%%%%%%%%%%%%%%%%%%%%%%%%%%%%%%%
%%%%%%%%%%%%%% CRAIG'S INFO  %%%%%%%%%%%%%
%%%%%%%%%%%%%%%%%%%%%%%%%%%%%%%%%

\author[C. J. Sutton]{Craig J. Sutton$^\sharp$}
\address{Dartmouth College\\ Department of Mathematics \\ Hanover, NH 03755}
\email{craig.j.sutton@dartmouth.edu}
\thanks{$^\sharp$ Research partially supported by an NSF Postdoctoral Fellowship
and NSF Grant DMS 0605247}


\begin{abstract}
We show that within the class of left-invariant naturally reductive
metrics $\mathcal{M}_{\operatorname{Nat}}(G)$ on a compact simple Lie group $G$,
every metric is spectrally isolated. 
We also observe that any collection of isospectral compact symmetric spaces is
finite; 
this follows from a somewhat stronger statement involving only a finite part of
the spectrum.
\end{abstract}

\maketitle

%%%%%%%%%%%%%%%%%%%%%%%%%%%%%%%%%%%%%%%%%%%%%%%%%%%%
%%%%%%%%%%%%%%%%%%  Introduction   %%%%%%%%%%%%%%%%%%%%%%%%%%
%%%%%%%%%%%%%%%%%%%%%%%%%%%%%%%%%%%%%%%%%%%%%%%%%%%%
\section{Introduction}

Given a connected closed Riemannian manifold $(M,g)$ its \emph{spectrum},
denoted $\Spec(M,g)$, is defined to be the sequence of eigenvalues $0 =
\lambda_{0} <
\lambda_{1} \leq \lambda_{2} \leq \cdots \nearrow + \infty$ of the associated
Laplacian $\Delta$ acting on smooth functions. Two Riemannian manifolds $(M_{1},
g_{1})$
and $(M_{2}, g_{2})$ are said to be \emph{isospectral} if their spectra
(counting multiplicities) agree.   Inverse spectral geometry is the
study of the extent to which geometric properties of a Riemannian manifold
$(M,g)$ are encoded in its
spectrum.

One typically expects that distinguished 
metrics, e.g., metrics of constant curvature or more generally symmetric spaces, 
should be spectrally distinguishable from other metrics.  However, only a very
few metrics, such as flat metrics on two-dimensional tori and Klein bottles as well as 
the round metric on the two-sphere, are 
actually known to be spectrally unique.
In fact, except among \emph{orientable} manifolds in low dimensions \cite{Tanno}, it is not known whether one can
tell from the spectrum whether a metric has constant sectional curvature or even
whether it is flat. Moreover, examples \cite{GSz} show that the spectrum does
not determine 
whether a closed manifold has constant scalar curvature or whether a manifold 
with boundary has constant Ricci curvature.

For metrics $g$ that are not spectrally determined or not known to be so, one
may ask about the size or structure of the isospectral set of $g$, i.e., of the
set of isometry classes of metrics that are isospectral to $g$. Types of results
include:
\begin{itemize}
\item {\bf Spectral finiteness}:  showing that a metric is determined up to
finitely
many possibilities, at least within some class of metrics.  For example,
isospectral
sets of Riemann surfaces \cite{McKean} and isospectral sets of flat tori
\cite{Wolpert, Pesce} are always finite.
\item {\bf Spectral isolation}:  showing that within the space of all metrics on
a
given manifold $M$, a punctured neighborhood of a metric $g$ (in a suitable
topology on the space of metrics) contains no metrics isospectral to $g$.  For
instance,
in all dimensions, the round metric on a sphere is spectrally isolated
\cite{Tanno2}, and any flat metric \cite{Ku} or any metric of constant negative
curvature \cite{Sh} on a manifold is spectrally isolated.
\item {\bf Spectral rigidity}: showing that a metric cannot be isospectrally
deformed. 
For example, negatively curved metrics \cite{GK, CS} 
and metrics of constant positive curvature \cite{Tanno2} are spectrally rigid.
\end{itemize}

For arbitrary metrics, a major open question is whether isospectral sets are
always compact.  More precisely, is it the case that every isospectral set of
closed Riemannian manifolds contains only finitely many diffeomorphism types
$M_1,\dots, M_k$ with corresponding sets of metrics
that are compact in the
$C^{\infty}$-topology?  Osgood, Phillips and Sarnak \cite{OPS} gave a positive
answer in the case of surfaces.  Under various restrictions, some progress has
been made in higher dimensions, but the general question remains elusive.  
Spectral isolation or rigidity results are not sufficient for compactness,
although they lend support to the conjecture.

Other than metrics of constant curvature, we are not aware of any special
metrics that are known to be spectrally isolated.   This article was initially
motivated by the difficult question of whether the spectrum determines whether a
compact Riemannian manifold  is symmetric.   As a first step, we consider the
following questions:\begin{enumerate} 
\item Is a symmetric Riemannian metric on a compact manifold spectrally isolated
within some larger natural class of metrics?   
\item Is every collection of mutually isospectral compact symmetric spaces
finite? 
\end{enumerate} 

We answer the second question affirmatively  in Corollary~\ref{cor1}.   In fact,
we
prove the following stronger statement:
\begin{result}
\emph{Within the class of compact symmetric spaces of a given dimension, each
compact symmetric space is finitely determined by a lower volume bound and a
finite part of its spectrum.}
\end{result}
\noindent
See Corollary~\ref{cor1} for a more precise statement.   Note that the various
known examples of isospectral flat tori show that the result cannot be
strengthened to spectral uniqueness in general.

For the first question, we consider the class of naturally reductive
left-invariant metrics on compact simple Lie groups.   The symmetric metrics
within this class are the bi-invariant metrics.   The naturally reductive
metrics are the
left-invariant metrics that in many ways most closely resemble the symmetric
ones.   For example, the geodesic symmetries are volume preserving, and every
geodesic is an orbit of a one-parameter group of isometries.   Moreover, while
the condition of natural reductivity does not imply the Einstein condition,
the class of naturally reductive metrics do provide a rich source of Einstein
metrics (see,
for example, \cite{DZ} or \cite{WangZiller}).  Thus one might expect that if the
bi-invariant metric is not uniquely determined by its spectrum, the naturally
reductive metrics
would be the most likely source of counterexamples.  We in fact find that not
only is the bi-invariant metric spectrally isolated
among the left-invariant naturally reductive metrics, but that \emph{every}
metric in this class is spectrally isolated within the class.

The class of left-invariant metrics on a compact Lie group $G$ of dimension $n$
may be identified with $\Symm_{+}(n, \R) \subset GL(n, \R)$, the space of
$n\times n$ positive definite symmetric matrices.   We can give the space
$\mathcal{M}_{\operatorname{Nat}}(G)$ of naturally reductive left-invariant
metrics the subspace topology from $\Symm_{+}(n, \R)$.

\begin{result}[Main Result, Theorem~\ref{Thm:LocRigid}]
\emph{Let $G$ be a compact simple Lie group.  Then
every metric $g$ in the space $ \mathcal{M}_{\operatorname{Nat}}(G)$ of
naturally reductive left-invariant metrics on $G$ is spectrally isolated
within $ \mathcal{M}_{\operatorname{Nat}}(G)$. That is, there exists an open
neighborhood $U$ of $g$ in
$\mathcal{M}_{\operatorname{Nat}}(G)$ such that if $g' \in U$ and $\Spec(G,g') =
\Spec(G,g)$, then $(G,g')$ and $(G,g)$ are isometric. }
\end{result}

\noindent
We do not know whether the spectral isolation statement of
Theorem~\ref{Thm:LocRigid} 
can be strengthened to uniqueness or finiteness. 
Currently, there are no known examples of pairs of isospectral left-invariant
naturally 
reductive metrics on Lie groups, but the second author has constructed examples
of isospectral naturally reductive
(in fact, normal homogeneous) metrics on simply-connected \emph{quotients} of
$\SU(n)$, where $n$ is extremely large \cite{Sut}.

Theorem \ref{Thm:LocRigid}
cannot be generalized to arbitrary left-invariant metrics.  Indeed, D. Schueth
\cite{Schueth}  and E. Proctor \cite{Proctor} have shown the existence of
non-trivial isospectral deformations
of left-invariant metrics on every classical simple compact Lie group except for
a few low-dimensional ones.  
On the other hand, the authors, along with D. Schueth, demonstrate in a
forthcoming article \cite{GSS}, that each \emph{bi-invariant} metric on a
connected compact Lie group is spectrally isolated within the class
of all left-invariant metrics.  

Presently, establishing the conjecture that the bi-invariant metric is
spectrally isolated
among \emph{all} metrics on a compact Lie group appears to be a rather
formidable task.
However, in light of the fact that most counterexamples in spectral geometry 
exploit metrics with ``large'' symmetry groups,  Theorem~\ref{Thm:LocRigid} and 
the results of \cite{GSS} lend strong support to this conjecture.  

The outline of this paper is as follows. In Section~\ref{Sec:BiInvariant} we
address the finiteness of isospectral sets of symmetric spaces. In
Section~\ref{Sec:NatRed} we review
the definition of naturally reductive metrics and the classification of
left-invariant naturally reductive metrics on compact simple Lie groups due to
D'Atri and Ziller \cite{DZ}, and we establish a structural result. 
Section~\ref{Sec:MainTheorem} contains the proof of our main result. 
We conclude our paper in Section~\ref{Sec:SemiNatRed}
by describing how to compute explicitly the spectra of certain left-invariant
naturally reductive metrics on Lie groups.   

\medskip
\noindent
{\bf Acknowledgments.} We consulted several colleagues concerning
Proposition~\ref{prop.conj} (1) and would like to thank G. Prasad for making us
aware of R.W. Richardson's paper ``A rigidity theorem for subalgebras of Lie and
associative algebras'' \cite{Richardson}, which provides an elegant proof.

%%%%%%%%%%%%%%%%%%%%%%%%%%%%%%%%%%%%%%%%%%%%%
%%%%%%%%  A Comment on Isospectral Sets of Bi-Invariant Metrics %%%%%%%%%%
%%%%%%%%%%%%%%%%%%%%%%%%%%%%%%%%%%%%%%%%%%%%%

\section{Isospectral sets of symmetric spaces}\label{Sec:BiInvariant}

\begin{dfn}\label{Def:ESet} Let $(M,g)$ be a closed Riemannian manifold.
\begin{enumerate}
\item[(i)] The \emph{eigenvalue set}, $\E(M,g)$, of $(M,g)$ is
the collection of eigenvalues of the Laplace-Beltrami operator $\Delta$ of
$(M,g)$
ignoring multiplicities.
\item[(ii)] The \emph{fundamental tone} of $(M,g)$, denoted by $\lambda_1(M,g)$,
is the lowest non-zero eigenvalue of the Laplace-Beltrami operator $\Delta$ of
$(M,g)$.
\end{enumerate}
\end{dfn}
 
 \begin{rem}\label{Rem:InvHomothety}
 We note that for any Riemannian manifold $(M,g)$ of dimension $n$ the quantity 
 $\lambda_1(M,g)^{\frac{n}{2}}\vol(M,g)$ is invariant under homotheties of $g$.
 \end{rem}
 
In this section, we focus on 
the class of all compact symmetric spaces, including not only symmetric
spaces of the so-called compact type but also those with Euclidean (toral)
factors, and we prove the following rigidity result:

 \begin{thm}\label{thm.sym2}  Given a positive integer $n$ and positive real
numbers $\lambda$ and $v$, there exists a constant $A>1$ depending only on $n$
and
on $\lambda^{\frac{n}{2}}v$ such that for any finite subset $E$  of the interval
$[\lambda,A\lambda]$, there exist up to isometry at
most finitely many $n$-dimensional compact symmetric spaces $(M,g)$ satisfying
the following
conditions: 
\begin{enumerate}
\item $\lambda_1(M,g)\geq \lambda$
\item $\vol(M,g)\geq v$
\item $\E(M,g)\cap [\lambda,A\lambda]\subset E$.
\end{enumerate}
\end{thm}

\begin{rem}
The choice of $A$ in the proof of the theorem will guarantee that $\E(M,g) \cap
[\lambda , A\lambda] \neq \emptyset$ 
for any symmetric space $(M,g)$ satisfying conditions (1) and (2) of the
theorem.
\end{rem}

\begin{cor}\label{cor1} 
In any given dimension, each compact symmetric space is finitely determined by a
lower volume bound and a finite part of its spectrum.
More specifically, let $M$ be a compact symmetric space of dimension $n$. 
Then, up to isometry, there exist only finitely many $n$-dimensional compact
symmetric spaces $M'$ such that 
\begin{enumerate}
\item $\vol(M')\geq \vol(M)$
\item $\E(M')\cap [0,A\lambda_1(M)] = \E(M)\cap [0,A\lambda_1(M)]$
\end{enumerate}
where $A$ is the constant given in Theorem~\ref{thm.sym2}.
\end{cor}
 
Since volume and dimension are spectral invariants, Corollary~\ref{cor1}
gives us the following result.
 
\begin{cor}\label{cor.sym} A compact symmetric space $(M,g)$ is finitely
determined by its spectrum within the class of compact symmetric spaces.
\end{cor}

We now turn to the proof of Theorem~\ref{thm.sym2}. 

\begin{notarem}\label{nota.sym}  \text{}
\begin{enumerate}  
\item[(i)] Every irreducible simply-connected
symmetric space $M$ of compact type is either a simple compact Lie group $G$
with a bi-invariant metric or it is of the form $G/K$, where $G$ is a simple
compact Lie group and $K$ is one of only finitely many connected Lie subgroups
of $G$ for which $G/K$ admits a symmetric metric.  We will consider both cases
together by taking $K=\{e\}$ in the first case.   The symmetric metric on $G/K$
is unique up to scaling.
 
 \item[(ii)] Every compact symmetric space is of the form 
 $$M=\Gamma\bs(M_0\times M_1\times\dots\times M_k)$$ where $M_0$ is a torus
(viewed as an abelian Lie group), $M_i$ is an irreducible symmetric space of
compact type, say $M_i=G_i/K_i$ as above, and $\Gamma$ is a
finite subgroup of the center of $M_0\times G_1\times\dots\times G_k$.  We may
assume that $\Gamma\cap M_0$ is trivial; otherwise, replace $M_0$ by the torus
$M_0/ (\Gamma\cap M_0)$.   Ignoring the metric,
we will refer to the structure $\Gamma\bs(M_0\times G_1/K_1\times\dots\times
G_k/K_k)$ as the \emph{homogeneity type} of $M$.  
Thus the homogeneity type of $M$ is its structure as a quotient of a compact Lie
group.

\item[(iii)]  Given $M$ with homogeneity type as above, let $\M$ be the
collection of symmetric metrics on $M$.  Each $g\in\M$ lifts to a symmetric
metric on $M_0\times M_1\times\dots\times M_k$ of the form $g_0\times
g_1\times\dots\times g_k$ where $g_i$ is a symmetric metric on $M_i$.  In
particular, $g_0$ is a flat metric on the torus $M_0$ and $g_i$ is given by a
multiple $-a_iB_{\mg_i}$ of the Killing form $B_{\mg_i}$ of the Lie algebra
$\mg_i$ of $G_i$, where $a_i > 0$.  Thus, the data defining the metric $g$
consists of the
positive scalars $a_1,\dots,a_k$ and the flat metric $g_0$.  Letting $\Gamma_0$
denote the image of $\Gamma$ under the homomorphic projection $M_0\times
G_1\times\dots\times G_k\to M_0$, then two metrics $g, g'\in\M$ are isometric if
the associated scalars satisfy $a_i=a_i'$, $1\leq i\leq k$ and if there exists a
$\Gamma_0$-equivariant isometry $(M_0,g_0)\to (M_0,g_0')$.
\end{enumerate}

 \end{notarem} 

\begin{lem}\label{lem.finhom}  In any given dimension $n$, there are only
finitely many homogeneity types of compact symmetric spaces.
\end{lem}

\begin{proof}  There are only finitely many homogeneity types of irreducible
simply-connected compact symmetric spaces of dimension at most $n$. The torus
$M_0$ (viewed as a Lie group) depends up to isomorphism only on its dimension.
Moreover for a given choice of $M_0\times M_1\times\dots\times M_k$ of dimension
$n$, there are only finitely many choices of $\Gamma$ as in \ref{nota.sym}. 
(Recall our standing assumption that $\Gamma\cap M_0=\{0\}$.)  
Indeed, the group $G:=G_1\times \dots\times G_k$ has finite center $Z(G)$ and
there are only finitely many homomorphisms from subgroups of $Z(G)$ into $M_0$. 
The lemma follows.
\end{proof}

 \begin{notarem}\label{nota.invar} \text{}
 Up to isometry, a flat torus may be expressed as $T=\Lc\bs\R^m$ with metric
given by the standard Euclidean inner product, where $\Lc$ is a lattice of full
rank in $\R^m$.  Two tori $\Lc\bs\R^m$ and $T=\Lc^{'} \bs\R^m$ 
are isometric if and only if $\Lc$ and $ \Lc'$ are congruent.
  Given $T=\Lc\bs\R^m$, let $\Lc^*$ be the dual lattice to $\Lc$.  We will refer
to the flat
torus $T^*=\Lc^*\bs\R^m$ (again with the standard Euclidean metric) as the
\emph{dual torus}. Recall that:
  \begin{enumerate}
 \item $\vol(T^*)=\frac{1}{\vol(T)}$.
 \item The spectrum of $T$ is given by the collection $4\pi^2\|\gamma\|^2$ as
$\gamma$ varies over $\Lc^*$.
 \item The systole $\syst(T)$, i.e., the length of the shortest non-contractible
loop in $T$, is given by the length of the shortest non-zero element of $\Lc$. 
In particular, $\lambda_1(T)=4\pi^2 \syst(T^*)^2$.
 \end{enumerate}
 \end{notarem}

\begin{sy}\label{Thm:Minkowski}(See \cite{Gromov} and \cite{GL}.)
There exist constants $C_m$ and $C_m'$ depending only on $m$ such that for all
flat tori $T=\Lc\bs\R^m$, we have:
\begin{enumerate}
\item $\syst(T) \leq C_m \vol(T)^{\frac{1}{m}}, $
where $C_m$ is a constant depending only on the dimension $m$. 
\item There exists at least one choice of lattice basis
$\{\delta_1,\dots,\delta_m\}$ of $\Lc$ satisfying the conditions that  
$\|\delta_1\|=\syst(T)$ and that
$$\prod_{i=1}^m\,\|\delta_i\|\leq C_m'\vol(T).$$
\end{enumerate}
\end{sy}

The first statement and the remarks in~\ref{nota.invar} imply:

\begin{cor}\label{cor.syst}  Let $T$ be a flat $m$-torus. Then
$$\lambda_1(T)\leq C_m\frac{1}{\vol(T)^{\frac{2}{m}}},$$
where $C_m$ is a constant depending only on the dimension $m$. 
\end{cor}

\begin{rem}
In \cite{Urakawa} it is shown that the corollary above fails for a non-abelian 
compact Lie group equipped with a left-invariant metric.
\end{rem}

\begin{lem}\label{lem.pes}  Theorem~\ref{thm.sym2} holds with ``compact
symmetric spaces'' replaced by ``flat tori''.
 \end{lem}
 
 Our proof of Lemma~\ref{lem.pes} closely follows H. Pesce's proof \cite{Pesce}
that there are at most finitely many tori with any given spectrum.

\begin{proof}[Proof of Lemma~\ref{lem.pes}.]  
A lattice $\Lc$ is determined up to congruence by 
$\|\alpha_i\|$, $i=1,\dots, n$ and $\|\alpha_i+\alpha_j\|$,
$1\leq i < j\leq n$, where $\{\alpha_1,\dots,\alpha_n\}$ is any lattice basis.
That is, if $\Lc'$ is another lattice, then $\Lc'$ is congruent to $\Lc$ if and
only if it admits a basis $\{\beta_1,\dots\beta_n\}$ with
$\|\alpha_i\|=\|\beta_i\|$ and $\|\alpha_i+\alpha_j\|=\|\beta_i+\beta_j\|$ for
all $i,j$. Since a pair of tori are isometric if and only if the dual tori
are isometric, a torus $T=\Lc\bs\R^n$ is determined up to isometry by the values
$4\pi\|\delta_i\|^2$, $1\leq i\leq n$ and $4\pi\|\delta_i+\delta_j\|^2$, $1\leq
i < j\leq n$ where $\{\delta_1,\dots, \delta_n\}$ is any lattice basis of
$\Lc^*$. 
Each of these values lies in the spectrum of $T$.  Choose the basis
$\{\delta_1,\dots, \delta_n\}$ of $\Lc^*$ as in Minkowski's Theorem so that
$4\pi\|\delta_1\|^2=\lambda_1(T)$ and $$\prod_{i=1}^n\,\|\delta_i\|\leq
C_n'\vol(T^*)=\frac{C_n'}{\vol(T)}.$$
Since $\|\delta_j\|\geq \|\delta_1\|$ for $j=1,\dots, n$, we then have for each
$i$ that 
$$\|\delta_i\|\leq C_n'\frac{1}{\|\delta_1\|^{n-1}\vol(T)}$$
and hence 
$$4\pi^2\|\delta_i\|^2\leq C\frac{1}{\lambda_1(T)^{n-1}\vol(T)^2},$$
where $C$ is a constant depending only on $n$.  
Replacing $C$ by $4C$, then the same bound is satisfied by 
$4\pi^2\|\delta_i+\delta_j\|^2$, $1\leq i <  j\leq n$. 
In particular, if $\lambda_1(T)\geq \lambda$ and $\vol(T)\geq v$, then
the elements $4\pi^2\|\delta_i\|^2$ and $4\pi^2\|\delta_i+\delta_j\|^2$, 
$1\leq i,j\leq n$ of $\Spec(T)$ all lie in the interval $[\lambda, A\lambda]$ 
where 
\begin{equation}\label{eq.A}
A= C\frac{1}{\lambda^{n}v^2}=C\frac{1}{(\lambda^{n/2}v)^2}.
\end{equation}
In particular, if $T$ satisfies the conditions of Theorem~\ref{thm.sym2}, then
each of $4\pi^2\|\delta_i\|^2$ and $4\pi^2\|\delta_i+\delta_j\|^2$ must lie in
the finite set $E$.  Thus, these values are determined up to finitely many
possibilities and hence so is the isometry class of $T$.
\end{proof}

\begin{proof}[Proof of Theorem~\ref{thm.sym2}.] 
By Lemma \ref{lem.finhom}, it suffices to prove the theorem with the
homogeneity type of $M$ fixed. We use the notation of~\ref{nota.sym}.  
Thus $M=\Gamma\bs(M_0\times M_1\times\dots\times M_k)$.  
Writing $M_i = G_i / K_i$, let $\Gamma_i$ be the  homomorphic image  
of $\Gamma\subset M_0\times G_1\times\dots\times G_k$ 
in $G_i$, $1\leq i\leq k$, and as in~\ref{nota.sym},
let $\Gamma_0$ be the homomorphic image of $\Gamma$ in $M_0$. 
Set $\omi =\Gamma_i\bs M_i$, $0\leq i\leq k$.  

In the notation of~\ref{nota.sym}, each  symmetric metric on $M$ is specified by
scaling factors $a_i$ (defining a metric $g_i$ on $M_i$), $i=1,\dots,k$, along
with a flat metric $g_0$ on $M_0$.   Given positive numbers $v$ and $\lambda$,
let $\M(\lambda,v)$ be the set of all such metrics satisfying
$\lambda_1(M,g)\geq\lambda$ and $\vol(M,g)\geq v$.  
Let $g\in\M(\lambda,v)$ and let $g_i$ denote the corresponding metric on $M_i$. 
Then the metric $g_i$ on $M_i$ induces a unique metric, denoted
$\overline{g_i}$, on $\omi$, $0\leq i\leq k$, 
so that $(M_i,g_i)\to(\omi,\overline{g_i})$ is a Riemannian covering. 
Since the projection $(M,g)\to (\omi,\overline{g_i})$ is a Riemannian submersion
with 
totally geodesic fibers it follows that $\E(\omi,\overline{g_i}) \subseteq
\E(M,g)$; 
hence, we see that  
\begin{equation}\label{eq.lo} 
\lambda_1(\omi,\ogi)\geq \lambda_1(M,g)\geq \lambda.
\end{equation}  

Write $V_i=\vol(\omi,\ogi)$, $m_i=\dim(M_i)$, and $n=\dim(M)=m_0+ m_1 + \cdots +
m_k$. 
For $i>0$, the metric $\ogi$ is uniquely determined by the scaling factor $a_i$
that
defines $g_i$.  By Remark~\ref{Rem:InvHomothety}, there exists a constant $c_i$,
independent of the scaling factor $a_i$ such that
\begin{equation}\label{eq.i}
\lambda_1(\omi,\ogi)^{\frac{m_i}{2}}V_i=c_i.
\end{equation}
Thus by Equation~\ref{eq.lo}, we have 
\begin{equation}\label{eq.i2}
V_i\leq \frac{c_i}{\lambda^{\frac{m_i}{2}}}.
\end{equation}
 By Corollary~\ref{cor.syst}, we similarly have 
\begin{equation}\label{eq.0}
V_0  \leq \frac{c_0}{\lambda^{\frac{m_0}{2}}},
\end{equation}
where $c_0$ is a constant depending only on the dimension $m_0$ of the torus
$M_0$.

Now 
\begin{equation}\label{eq.v}
\vol(M,g)=aV_0V_1\dots V_k,
\end{equation}
where 
$a=\frac{1}{|\Gamma|}|\Gamma_1| \cdots |\Gamma_k|$.
Since $\vol(M,g)\geq v$, Equations~\ref{eq.i2},~\ref{eq.0} and~\ref{eq.v} imply
for
$0\leq i\leq k$ that 
\begin{equation}\label{eq.all}
V_i\geq\frac{v}{aV_1\cdots \widehat{V}_i \cdots V_k}\geq 
\frac{c_i v\lambda^{\frac{n}{2}}}{ac\lambda^{\frac{m_i}{2}}}, 
\end{equation}
where $c=c_0c_1\cdots c_k$.  

For $1\leq i\leq k$, it follows from Equations~\ref{eq.i} and~\ref{eq.all} that 
$$(\frac{\lambda_1(\omi,\ogi)}{\lambda})^{\frac{m_i}{2}}\leq \frac{ac}{
v\lambda^{\frac{n}{2}}}$$
and thus 
\begin{equation}\label{eq.i3}
\frac{\lambda_1(\omi,\ogi)}{\lambda}\leq
(\frac{ac}{v\lambda^{\frac{n}{2}}})^{\frac{2}{m_i}}.
\end{equation}
Let $A_i$ be the expression on the right hand side of Equation~\ref{eq.i3}.
Then 
\begin{equation}\label{eq.ai}
\lambda_1(\omi,\ogi)\in [\lambda, A_i\lambda].
\end{equation}

Next consider $(\omn,\ogn)$. We apply Lemma~\ref{lem.pes}, with $v$ replaced by
the lower bound for the volume of $V_0$ given in Equation~\ref{eq.all} (with
$i=0$).  Note that $m_0$ plays the role of $n$ in the lemma.  We continue to use
$\lambda$ for the lower bound on $\lambda_1$. Letting $A_0$ be the expression in
Equation~\ref{eq.A} with our new lower bound on volume, we have
\begin{equation}\label{eq.a0}
A_0= \tilde{C}\frac{\lambda^{m_0}}{\lambda^{m_0}v^2\lambda^n}= 
\frac{\tilde{C}}{(\lambda^{n/2}v)^2},
\end{equation}
where $\tilde{C}$ is a constant depending only on the homogeneity type of $M$.  

Letting $$A=\max\{A_0,A_1,\dots, A_k\},$$ then the discussion above shows us
that 
for any $g\in\M(\lambda,v)$ we have $\E(M,g) \cap [\lambda , A \lambda] \neq
\emptyset$.
Now, let $E$ be as in the statement of the theorem and let
$\M(\lambda,v, E)$ be the collection of all symmetric metrics
$g\in\M(\lambda,v)$ such that $\E(M,g)\cap[\lambda, A\lambda]\subset E$.    For
each such $g$, we have by Equation~\ref{eq.ai} that $\lambda_1(\omi,\ogi)\in E$,
$1\leq i\leq k$.  Since $E$ is finite, there are only finitely many
possibilities for $\lambda_1(\omi,\ogi)$ and thus only finitely many
possibilities for $a_i$.  Next by Lemma~\ref{lem.pes} and the fact that $A\geq
A_0$, the isometry class of $\ogn$ is determined up to finitely many
possibilities. An elementary argument then shows that the $\Gamma_0$-equivariant
isometry class of $(M_0,g_0)$ is determined up to finitely many possibilities.
 The theorem now follows from the remarks in~\ref{nota.sym}(iii). 

\end{proof} 

%%%%%%%%%%%%%%%%%%%%%%%%%%%%%%%%%%%%%%%%%
%%%%%%%%%%%%%   Naturally Reductive Metrics  %%%%%%%%%%%%
%%%%%%%%%%%%%%%%%%%%%%%%%%%%%%%%%%%%%%%%%

\section{Naturally Reductive Metrics }\label{Sec:NatRed}

Let $(M,g)$ be a connected homogeneous Riemannian manifold.   Choose a base
point $o\in M$.  Let $H$ be a transitive group of isometries of $(M,g)$, and let
$K$ be the isotropy subgroup at $o$.  The Lie algebra $\mh$ of $H$  decomposes
into a sum $\mh=\germ{K}+\mq$, where $\germ{K}$ is the Lie algebra of $K$ and
$\mq$ is an $\Add(K)$-invariant complement of $\germ{K}$.  Given a vector $X \in
\mh$ we obtain a Killing field
$X^*$ on $M$ by 
$X_{p}^*  \equiv \frac{d}{dt}|_{t=0} \exp{tX}\cdot p$ for $p\in M$.
The map $X \mapsto X^*$ is an antihomomorphism of Lie algebras.
We may identify
$\mq$ with $T_{0}M$ by the linear map $X \mapsto X_{o}^*$.   Thus the
homogeneous Riemannian metric $g$ on $M$  corresponds to an inner product
$\langle \cdot , \cdot
\rangle$ on $\mq$.   For $X\in\mg$, write $X=X_\germ{K}+ X_\mq $ with
$X_\germ{K}\in \germ{K}$ and $X_\mq\in\mq$.   Recall that for $X,Y\in\mq$, 
$$(\nabla_{X^*}Y^*)_{o} = -\frac{1}{2}([X,Y]_{\mq}^*)_o + U(X,Y)^*_o,$$
where $U: \mq \times \mq \to \mq$ is defined by 
$$2<U(X,Y), Z> = \langle [Z, X]_{\mq}, Y \rangle + \langle X, [ Z,
Y]_{\germ{p}} \rangle.$$

\begin{dfn}\label{def.natred}
Let $(M,g)$ be a Riemannian homogeneous space and let $H$ be a transitive group
of isometries of $(M,g)$.
$(M,g)$  is said to be {\bf naturally reductive} (with respect to $H$), if
there exists an $\Ad(K)$-invariant complement $\mq$ of $\germ{K}$ (as above)
such that 
$$ \langle [Z, X]_\mq, Y \rangle + \langle X, [Z,Y]_\mq\rangle
=0,$$
or equivalently $U \equiv 0$. That is, for any $Z \in \mq$  the map $[Z,
\cdot ]_{\mq} : \mq \to \mq$ is skew symmetric with respect to
$\langle \cdot , \cdot \rangle$.
\end{dfn}

\begin{rem}\text{}
Naturally reductive metrics are a generalization of symmetric metrics.
Although the geodesic symmetries of naturally reductive metrics need not be
isometries, they are (up to sign) volume preserving.   Moreover, every geodesic
is an orbit of a one-parameter subgroup of the transitive group $H$.  In
particular, the
geodesics through the base point $o$ are precisely the curves $\exp(tX)\cdot~o$,
$X\in\mq$.
\end{rem}

\begin{thm}[\cite{DZ} Theorems 3 and 7]\label{thm:NatRed}
Let $G$ be a connected compact simple Lie group and let $g_0$ be the
bi-invariant Riemannian metric
on $G$ given by the negative of the Killing form. Let $K \leq G$ be a connected
subgroup with Lie algebra $\germ{K}$,  and let $\mp$ be a 
$g_{0}$ orthogonal complement of $\germ{K}$ in $\mg$.  Given any $\alpha\in \R$
and any bi-invariant Riemannian metric $h$ on $K$, define a left-invariant
metric $g$ on $G$ by the orthogonal direct sum
\begin{eqnarray}\label{Eq:NatRed}
g &=& e^{\alpha} g_{0} \upharpoonright \germ{p} \oplus h \upharpoonright
\germ{K}.\end{eqnarray}
 Then:
\begin{enumerate}
\item $g$ is $\Add(K)$-invariant and is naturally reductive with respect to
$G\times K$, where $G$ acts by left translations and $K$ by right translations
on $G$.

\item $g$ is normal homogeneous if and only if $h \leq e^{\alpha}g_0
\upharpoonright \germ{K}$.

\item Let $N_G(K)$ denote the normalizer of $K$ in $G$ and $N_G(K)^0$ denote the
connected component of the identity. Then the full connected compact isometry
group of $(G,g)$ is given by $G \times N_G(K)^0$, where $G$ acts via left
translations and $N_G(K)^0$ acts via right translations on $G$. The metric $g$
is therefore $N_G(K)$ bi-invariant and we conclude that the left cosets of
$N_G(K)^0$ are totally geodesic submanifolds of $(G,g)$.

\item Every left-invariant naturally reductive metric on $G$ is of the form
given in Equation~\ref{Eq:NatRed} for some  $K$, $\alpha$, and $h$.
\end{enumerate}
 \end{thm}

\begin{notarem}\label{rem.natred} \text{}

\begin{enumerate}
\item\label{rem.MetricNot} We will denote a metric $g$ of the form given in
Equation~\ref{Eq:NatRed} by $g_{\alpha, h}$.

\item\label{rem.AltNatRed}  In the setting of Theorem \ref{thm:NatRed}, it
follows from the second statement of the theorem that the metric $g$ is also
naturally reductive with respect to $G\times L$ where $L$ is any connected
subgroup of $N_G(K)$ containing $K$. 

\item\label{rem.ConjNatRed}  If $g$ is naturally reductive with respect to
$G\times K$, then it is also naturally reductive with respect to $G\times
aKa^{-1}$ for any $a\in G$. 
Conjugating $K$ corresponds to changing the choice of base point for which $K$
is the isotropy group.

\item A Lie group $G$ can admit metrics naturally reductive with respect to $H
\times K$ where $H, K < G$, but which are not left-invariant \cite[p.
12--14]{DZ}. Such metrics are sometimes called semi-invariant.

\item\label{rem.GenNatRed} If $G$ is an arbitrary connected compact Lie group it
is known that left-invariant metrics of the form given by
Equation~\ref{Eq:NatRed} are naturally reductive, where we allow 
$g_0$ to denote \emph{any} bi-invariant metric on $G$.  (In the case that $G$ is
simple as in Theorem ~\ref{thm:NatRed}, all bi-invariant metrics are multiples
of each other.   Since $\alpha$ is arbitrary, no greater generality is achieved
by varying $g_0$.)  However, it is unknown
whether (up to isometry) this list is exhaustive (see \cite[Theorem 1 and p.
20]{DZ}). 
\end{enumerate}

\end{notarem}

\begin{prop}\label{prop.conj} Let $G$ be a compact Lie group.  Then
\begin{enumerate}
\item There are only finitely many conjugacy classes of semisimple subgroups of
$G$.

\item There are only finitely many conjugacy classes of subgroups of $G$ of the
form $TK$ where $K$ is trivial or connected and semisimple, and $T$ is a maximal
torus in
the centralizer of $K$ in $G$.
\end{enumerate}
\end{prop}

\begin{proof}  The first statement follows directly from the fact
that up to conjugacy a real (or complex) Lie algebra contains finitely many
semi-simple Lie algebras (cf. \cite[Prop. 12.1]{Richardson}).
The second statement follows from the first statement, the fact
that the centralizer of any closed subgroup $K$ of $G$ is compact, and the
conjugacy of all maximal tori in any compact Lie group.
\end{proof}

\begin{cor}\label{cor.finite} Let $G$ be a connected simple compact Lie group. 
Then there exists a finite collection $\mathcal{K}$ of connected subgroups of
$G$ such that (up to isometry) every naturally reductive left-invariant metric
$g$ on $G$ is
naturally reductive with respect to $G\times K$ for some $K\in\mathcal{K}$.  The
collection $\mathcal{K}$ consists of a choice of representative from each of the
conjugacy classes of subgroups of $G$ given in Proposition \ref{prop.conj}(2).
Hence, the moduli space of naturally reductive metrics
$\mathcal{M}_{\operatorname{Nat}}(G)$ is given by 
$$\mathcal{M}_{\operatorname{Nat}}(G) = \cup_{K \in \mathcal{K}}
\mathcal{M}_{K}$$
where $\mathcal{M}_{K}$ is the space of metrics on $G$ naturally reductive with
respect to $G\times K$.  

\end{cor}

\begin{proof}  Choose a connected subgroup $H$ of $G$ such that $g$ is of the
form given in Theorem \ref{thm:NatRed} with respect to $G\times H$.  Write
$H=H_{z}H_{ss}$, where $H_z$ is abelian and  $H_{ss}$ is semi-simple or trivial. 
Let $T$
be a maximal torus in the centralizer of $H_{ss}$ containing $H_z$, and let $K=
T H_{ss}$.  Note that $K$ is a connected subgroup such that $H\subset K\subset
N_G(H)$.  
Thus by Remark \ref{rem.natred}(\ref{rem.AltNatRed}), $g$ is naturally reductive
with respect to $G\times K$.  Finally, observe that, up to conjugacy, $K$ lies
in $\mathcal{K}$, so the corollary follows from Remark
\ref{rem.natred}(\ref{rem.ConjNatRed}).

\end{proof}

%%%%%%%%%%%%%%%%%%%%%%%%%%%%%%%%%%%%%%%%%%
%%%%%%%%%%%%   Proof of the Main Theorem  %%%%%%%%%%%%%%
%%%%%%%%%%%%%%%%%%%%%%%%%%%%%%%%%%%%%%%%%%

\section{Proof of the Main Theorem}\label{Sec:MainTheorem}

We are ready to prove the following theorem.

\begin{thm}\label{Thm:LocRigid}
Let $G$ be a connected compact simple Lie group and let
$\mathcal{M}_{\operatorname{Nat}}(G)$
denote the class of left-invariant naturally reductive metrics on $G$, where
$\mathcal{M}_{\operatorname{Nat}}(G) \subset GL(n,\R)$ has the subspace
topology. Then every metric $g \in \mathcal{M}_{\operatorname{Nat}}(G)$ is
spectrally isolated within this class. That is, there exists an open
neighborhood $U$ of $g$ in $\mathcal{M}_{\operatorname{Nat}}(G)$ such that if
$g' \in U$ and $\Spec(G,g') = \Spec(G,g)$, then $(G,g')$ and $(G,g)$ are
isometric. 
\end{thm}

\begin{lem}\label{lem.al}  Let $G$ be a connected compact simple Lie group. 
Using the notation of Theorem \ref{thm:NatRed} and of Corollary
\ref{cor.finite}, let  $K\in\mathcal{K}$ and  let $g=g_{\alpha,h}\in\mc_K$. 
Then there exists a neighborhood $W$ of $g$  in $\mc_K$ 
such that if $g'=g_{\alpha',h'}\in\mc_K$ is
isospectral to $g$, then $\alpha'=\alpha$.  
\end{lem}

\begin{proof}  Let $\pi:G\to G/K$ be the canonical projection.  For
$g'\in\mc_K$, the metric $\gb '$  on $G/K$ induced
by $g'$ depends only on $\alpha'$ and is a scalar multiple of the metric $\gb$
induced by $g$.     Since $g'$ is $K$-bi-invariant, the Riemannian submersion
$\pi:(G,g')\to (G/K,\gb')$ has totally geodesic fibers.  Thus the spectrum of
$(G/K,\gb')$ is contained in the spectrum of $(G,g')$. 
Let $\lambda$ be the lowest non-zero eigenvalue of $\Spec(G/K,\gb)$.  Choose
$\epsilon>0$ so that $\Spec(G,g)$ has no entries in the interval
$(\lambda-\epsilon,\lambda+\epsilon)$ other than $\lambda$.  For $g'\in\mc_K$,
the spectrum
of  $\gb'$ is just a rescaling of
$\Spec(G/K,\gb)$ with the scale factor depending non-trivially on $\alpha'$. 
Thus, we
can find $\delta>0$ so that if $0<|\alpha'-\alpha|<\delta$, then the lowest
non-zero eigenvalue of $\gb'$ lies in the interval 
$(\lambda-\epsilon,\lambda+\epsilon)$ and is distinct from $\lambda$.  But then
$g'$ cannot be isospectral to $g$.  The lemma follows.
\end{proof}

\begin{nota}\label{nota.h}  Let $G$ be a compact, simple, connected Lie group,
and let $g_0$ be the negative of the Killing form.  Given a connected compact
Lie subgroup $K$, let $\mb(K)$ denote the set of all bi-invariant metrics on
$K$.   For $h\in \mb(K)$, write $h(X,Y)=g_0(A_hX,Y)$.  The map $h\mapsto A_h$ is
a bijection between $\mb(K)$ and the set of all linear transformations
of $\germ{K}$ that commute with $\Add(K)$ and are positive-definite symmetric
(with respect to the restriction of $g_0$ to $\germ{K}$). We give $\mb(K)$ the
subspace
topology from the space of positive-definite symmetric linear transformations.  

For $h\in\mb(K)$, we will say that $\beta\in\R$ is {\bf \emph{admissible}} for
$h$ if
$e^\beta$ is strictly larger than each of the eigenvalues of $A_h$.  Given
$\beta\in\R$, we let $\Adm(K, \beta)$ be the set of all $h\in\mb(K)$ such that
$\beta$ is admissible for $h$.  Note that $\Adm(K, \beta)$ is an open subset of
$\mb(K)$,  

\end{nota}

\begin{prop}\label{prop.beta}    We use the notation of Corollary
\ref{cor.finite}, and Notation \ref{nota.h}.  Let  $K\in\mathcal{K}$ and let
$\beta\in\R$.  Define $\Phi: K \times K\to K$ by
$\Phi(k_1,k_2)=k_1k_2^{-1}$.  Then:

\begin{enumerate}
\item For each  $h\in \Adm(K,\beta)$, there exists a unique metric
$\hh\in\mb(K)$ for which $$\Phi:(K \times K, e^\beta g_0 \times \hh) \to (K,h)$$   
is a Riemannian submersion.
\item The mapping $F:\Adm(K, \beta)\to \mb(K)$ given by $h\mapsto\hh$ is a
homeomorphism from $\Adm(K, \beta)$ to an open subset of $\mb(K)$.
\end{enumerate}
\end{prop}

\begin{proof}  For $\Phi$ as in the  Proposition, the kernel of $\Phi_*$ is
given by the diagonal subspace $\Delta \germ{K} = \{(X,X):X\in\mk\}$.  For
$\hh\in\mb(K)$, the orthogonal complement $\mathcal{H}$ of $\Delta \germ{K}$ in
$\germ{K}\times\germ{K}$ with respect to $e^\beta g_0\oplus \hh $ is given by 
$$\mathcal{H}= \{(X,Y)\in\germ{K}\times\germ{K}: e^\beta X +A_{\hh}
Y=0\}=\{\tilde{X}:=(X,-e^\beta A_{\hh}^{-1}X):X\in\mk\}.$$  
Denoting the left-invariant metric $e^\beta g_0\oplus \hh$ by $q$, a simple
computation shows that for $\tilde{X},\tilde{Y}\in\mathcal{H}$, we have 
$$q(\tilde{X},\tilde{Y})=g_0(X,e^\beta DY)$$
where $D=Id +e^\beta A_{\hh}^{-1}$.  

Next $\Phi_*(\tilde{X})=X+e^\beta A_{\hh}^{-1}(X) =
DX$.  Thus  $\Phi_*:(\mathcal{H}, q)\to (\germ{K}, h)$
is an inner product space isometry if and only if  $DA_hD=e^\beta D$, i.e.,
$D=e^{\beta}A_h^{-1}$. 
The admissibility of $\beta$ implies that this transformation $D$ has all
eigenvalues strictly larger than one.  Thus given $h$, we may define $\hh$ by
the condition
$A_{\hh} =e^\beta(D-Id)^{-1}=e^\beta(e^{\beta}A_h^{-1}-Id)^{-1}$.  Note that
$A_{\hh}$ is positive-definite and symmetric and it commutes with $\Add(K)$, so
$\hh$ is a well-defined element of $\mb(K)$.   This proves statement (1). 
Statement (2) is also immediate.

\end{proof}

\begin{cor}\label{cor.beta} Let $G$ be a connected compact simple Lie group. 
Using the notation of Theorem \ref{thm:NatRed} and of Proposition
\ref{prop.beta}, let  $K\in\mathcal{K}$, let $\mp$ be an $\Add(K)$-invariant
complement of $\germ{K}$ in $\mg$,  and  let $\alpha,\beta\in\R$.  Let
$\tg=\tg_{\alpha,\beta}$ be the left-invariant naturally reductive Riemannian
metric on $G$ given by 
$$\tg=e^\alpha g_{0} \upharpoonright \mp \oplus e^\beta g_{0} \upharpoonright
\mk.$$  For $h\in \Adm(K, \beta)$,
define $\hh=F(h)$ as in Proposition~\ref{prop.beta}.  Then the map 
$$\Phi:(G \times K,\tg \times \hh) \to (G,g_{\alpha,h})$$ 
given by $(g,k)\mapsto
gk^{-1}$ is a Riemannian submersion.  Moreover, the fibers of this submersion
are totally geodesic.
\end{cor}

\begin{proof}  The first statement follows easily from
Proposition~\ref{prop.beta}.  The assertion that the fibers are totally geodesic
follows from the facts that the metric on $G\times K$ is bi-invariant with
respect to the subgroup $K\times K$ and that the fibers are cosets in $G\times
K$ of $\Delta K \leq K\times K$.
\end{proof}

\begin{rep}\label{rept}  We briefly review the representation theory needed in
the proof of the main theorem.

(i)  If $G$ is a compact Lie group let $\widehat{G}$ denote the set of
equivalence classes of irreducible unitary representations of $G$. 
Since $G$ is compact we see that all of its irreducible representations are
finite dimensional.   
For $\mu\in\widehat{G}$ and $\sigma$ any unitary representation of $G$, let
$[\sigma:\mu]$ denote the  multiplicity of $\mu$ in $\sigma$.  (More generally,
for any representations $ \rho: G \to U(\mH_{\rho})$ and $\tau: G \to
U(\mH_{\tau})$, one defines 
$$[\rho: \tau] = \dim \{ T \in \Hom(\mH_{\tau} , \mH_{\rho}) : T \circ \tau(g) =
\rho(g) \circ T \mbox{ for all } g \in G \}.$$
The {\bf \emph{right regular representation}} of $G$ is the mapping $\rho_G: G
\to U(L^2(G))$, where 
$$\rho_G(g)f(x) = f(xg)$$
for any $f \in L^2(G)$.
The Peter-Weyl theorem (see \cite[p. 133]{Folland}) states that every
irreducible unitary
representation $\pi : G \to U(\mathcal{H}_{\pi})$ occurs in the right regular
representation $\rho_G$ of $G$ on $L^2(G)$; in fact $[\rho_G:\pi]=d_\pi$ where
$d_\pi$ is the dimension of $\pi$.    

 (ii) Let $G$ and $K$ be compact Lie groups.  Given two representations $\sigma:
G \to U(\mathcal{H}_{\rho})$ and $\tau: K \to U(\mathcal{H}_{\tau})$  we may
define the {\bf \emph{Kronecker product}} $\sigma \otimes \tau$ to be the
unitary
representation of $G \times K$ on $\mH_{\sigma} \otimes \mH_{\tau}$ given by
$\sigma\otimes\tau(g,k)(v\otimes w)=\sigma(g)v\otimes \tau(h)w$.  It is a
standard fact that the irreducible representations of $G\times K$ are precisely
the representations of the form $\sigma \otimes \tau$ where $\sigma \in
\widehat{G}$ and $\tau \in \widehat{K}$ and we can identify $\widehat{G \times
K}$ with $\widehat{G} \times \widehat{K}$ (see \cite[Theorem 7.25]{Folland}).
Hence, by the Peter-Weyl theorem,
$\sigma\otimes \tau$ occurs in the right regular representation $\rho_{G\times
K}$ of $G\times K$ on $L^2(G \times K)$.
In the case of finite-dimensional representations, we may identify $\mH_{\sigma}
\otimes \mH_{\tau}$ with $ \Hom(\mH_{\tau}^{*}, \mH_{\sigma})$ and then
\begin{equation}\label{eq.int}\sigma\otimes \tau (g,k) T := \sigma(g) \circ
T
\circ \bar{\tau}(k^{-1}),\end{equation}
for $T \in \Hom(\mH_{\tau}^{*}, \mH_{\sigma})$, where $\overline{\tau} :G\to
U(\mH_{\tau}^*)$ is the contragredient  of $\tau$ given by $\overline{\tau}(g) S
= S \circ \tau(g^{-1})$ for $g \in G$ and $T \in \mH_{\rho}^{*}$.  In the case
of interest to us, $K$ will be a compact subgroup of
$G$.  Let $\Delta K$ denote the diagonal subgroup of $G\times K$.  
Observe that if $T \in \Hom(\mH_{\tau}^{*}, \mH_{\sigma})$ is fixed by
$\sigma\otimes \tau(\Delta K)$,
then Equation \ref{eq.int} says that $T$ intertwines the contragredient
$\overline{\tau}$ of $\tau$ with the restriction of $\sigma$ to $K$.    In
particular, if $\tau$ is irreducible (and hence necessarily finite-dimensional),
then $\mH_{\sigma} \otimes \mH_{\tau}$ contains a non-trivial $\Delta K$-fixed
vector if and only if
the restriction of $\sigma$ to $K$ contains a  copy of $\overline{\tau}$; that
is, $[\sigma \upharpoonright K : \overline{\tau}] \neq
0$.\label{Fact:DeltaFixed2}

(iii) Let  $K$ be a connected compact Lie subgroup of the connected compact Lie
group $G$.  For $\tau\in\widehat{K}$, let $\Ind^{G}_{K}(\tau)$ denote the
representation of $G$ induced by $\tau$  (see \cite[p. 152]{Folland}), and for
$\sigma\in\widehat{G}$, let $\sigma \upharpoonright K$ denote the restriction of
$\sigma$ to $K$.    Then Frobenius' reciprocity theorem (see \cite[p.
160]{Folland}) 
shows us that $$[\sigma \upharpoonright K : \tau] = [
\Ind^{G}_{K}(\tau):\sigma].$$ 

(iv)   Again letting  $K$ be a connected compact Lie subgroup of the connected
compact Lie group $G$, every irreducible unitary representation $\tau$ of $K$
occurs in $\rho_G\upharpoonright K$.  Indeed, let $\sigma \in \widehat{G}$ be an
irreducible representation such that $[\Ind^{G}_{K}(\tau) : \sigma] \neq 0$.
Then by (iii), we have $[\sigma \upharpoonright K : \tau] \neq 0$ and thus by
(i), $[\rho_G\upharpoonright K:\tau]\neq 0$.

(v) Given a left-invariant Riemannian metric $g$ on a compact Lie group $G$,
then every 
$\rho_G$ invariant subspace of $L^2(G)$ is also invariant under the action of
the Laplacian
 of $G$.   If the metric is bi-invariant, then the Laplacian commutes with the
action $\rho_G$ and thus each $\rho_G$-irreducible subspace of $L^2(G)$ lies in
a single eigenspace.  More generally, if the metric is left $G$-invariant and
right $K$-invariant for some subgroup $K$ of $G$, then each
$\rho_G\upharpoonright K$ irreducible subspace of $L^2(G)$ lies in a single
eigenspace.
\end{rep}

\begin{invar}\label{nota.spec}  Let $K$ be a compact connected Lie group, and
let $\mathcal{B}(K)$ denote the set of all bi-invariant metrics on $K$.   We
define a complete set of invariants of the elements of $\mathcal{B}(K)$.

First, if $K$ is simple and $h\in\mb(K)$, let $\lambda_1(h)$ denote the lowest
nonzero
eigenvalue of the Laplacian of $h$.  Since all elements of $\mb(K)$ are scalar
multiples of each other, the map $\lambda_1:\mb(K)\to\R_+$ is bijective. 

Second, if $K$ is a torus, choose
and fix a Lie group isomorphism of $K$ with $\Z^m\bs\R^m$, thus realizing each
bi-invariant (i.e., flat) metric $h\in\mb(K)$
as an inner product on $\R^m$. (We note that this is different from 
our approach in Section~\ref{Sec:BiInvariant}, where we changed the geometry of
the torus by varying the lattice).  
Let $h^*$ be the dual inner product on the dual space $(\R^m)^*$, and 
let $Q^*$ be the associated quadratic form $Q^*(\gamma)=h^*_0(\gamma,\gamma)$.  
Letting $\{\delta_1,\dots,\delta_m\}$ be the standard basis of 
$(\R^m)^*$, set $b_j(h):=4\pi^2 Q^*(\delta_j)$ and $c_{jk}(h)=4\pi^2
Q^*(\delta_j+\delta_k)$ 
for $1\leq j<k\leq m$.   Then, as pointed out in \cite{Pesce}, $Q^*$, and thus
$h$, is
uniquely determined by the invariants $b_j(h)$ and $c_{jk}(h)$, $j,k=1,\dots,m$. 

For the general case, write $K=K_0K_1\dots K_r$, where $K_0$ is a
torus (the identity component of the center of $K$) and $K_1,\dots, K_r$ are
simple normal subgroups.  The various $K_i$'s may intersect but only in finite
central subgroups.  The homomorphic projections $\mk\to\mk_i$ give rise to
homomorphisms $\pi_i: K\to \kb_i$, where  $\kb_i$ is a compact Lie group
finitely covered by $K_i$ for $0\leq i\leq r$.  
Each bi-invariant metric $h$ on $K$ induces bi-invariant metrics $\qb_i$  on
$\kb_i$, and,  conversely, the  bi-invariant metrics $\qb_0,\qb_1,\dots, \qb_r$
uniquely
determine $h$.   For $i=1,\dots, r$, define $\lambda^i_1:\mb(K)\to\R_+$ by
$\lambda^i_1(h)=\lambda_1(\qb_i)$.   Choose
and fix a Lie group isomorphism of $\kb_0$ with $\Z^m\bs\R^m$ where
$m=\dim(K_0)$, and define $b_j(q)=b_j(\qb_0)$ and $c_{jk}(q)=c_{jk}(\qb_0)$ for
$j,k=1,\dots, m$.  

For notational convenience, we rename the full collection of invariants as
$\gamma_p:\mb(K)\to\R$, $p=1,\dots, s$, where $s=r+m+{m\choose 2}$, $r$ is the
number of simple factors in $K$ and $m$ is the dimension of the center of $K$.  

\end{invar}

\begin{prop}\label{prop.invar}  In the notation of \ref{nota.spec}, we have:

\begin{enumerate}
\item The collection of invariants $\gamma_p(h)$, for $1\leq p\leq s$, depend
continuously on $h$ and form a complete set of invariants of the elements of
$\mb(K)$. 
\item Each of these invariants lies in $\Spec(K,h)$.  Moreover, for each $p$,
the $\gamma_p(h)$-eigenspace contains a  $\rho_K$-invariant, irreducible
subspace $\mathcal{H}_p\subset L^2(K)$ that is independent of $h\in\mb(K)$. 
\item Each element of $\mb(K)$ is spectrally isolated within $\mb(K)$.
\end{enumerate}
\end{prop}

\begin{proof}  The first  two statements are straightforward consequences of
\ref{nota.spec}.   The third statement follows from the first two.
\end{proof}

\begin{lem}\label{lemma.rep} Let $G$ be a connected, compact simple Lie group
and let $K\in\mathcal{K}$.  Let $\alpha, \beta\in\R$.  For each $h\in
\Adm(K,\beta)$, let $\hh=F(h)$ as in Proposition~\ref{prop.beta} so that we have
a Riemannian submersion 
$\Phi:(G \times K,\tg_{\alpha,\beta} \times \hh ) \to (G,g_{\alpha,h})$ 
given by $\Phi(x,y)=xy^{-1}$.  (See Corollary~\ref{cor.beta}.) 
In the notation of \ref{nota.spec}, for each $p=1,\dots, s$, there exists
$\zeta_p\in
\Spec(G,\tg_{\alpha,\beta})$ such that $\zeta_p+ \gamma_p(\hh)\in
\Spec(G,g_{\alpha,h})$ for all $h\in\Adm(K,\beta)$.
\end{lem}

\begin{proof}  
Since the Riemannian submersion 
$\Phi:(G \times K,\tg_{\alpha,\beta} \times \hh) \to (G,g_{\alpha,h})$
has totally geodesic fibers given by (left) translates of
$\Delta K \leq G \times K$, we see that the spectrum of $(G,g_{\alpha,h})$
coincides with the spectrum of the Laplacian of $(G \times K,\tg_{\alpha,\beta}
\times \hh) \to (G,g_{\alpha,h})$ restricted to $\Delta K$-invariant functions
on $L^2(G \times K)$.  
Let  $\tau_p$ be the irreducible representation of $K$ on
$\mathcal{H}_p<L^2(K)$ given in Proposition~\ref{prop.invar} (2).    By
\ref{rept}(iv), the contragredient
$\overline{\tau}$ occurs in $\rho_G\upharpoonright K$, say with representation
space $\mathcal{H}_{\overline{\tau}} <L^2(G)$.   Since the metric
$\tg_{\alpha,\beta}$ on $G$ is $K$-bi-invariant, we have by \ref{rept}(v) that
$\mathcal{H}_{\overline{\tau}}$ lies in an eigenspace of the Laplacian of
$(G,\tg_{\alpha,\beta})$, say with eigenvalue $\zeta_p$.  By \ref{rept}(ii), 
there is a non-trivial $\Delta K$-fixed vector in $\mH_{\overline{\tau}}\otimes
\mH_p < L^2(G\times K)$.   This vector is a $(\zeta_p +
\gamma_p(\hh))$-eigenfunction on  $(G \times K,\tg_{\alpha,\beta} \times \hh)$; 
it is $\Delta K$-invariant and thus descends to an eigenfunction on
$(G,g_{\alpha,h}).$
\end{proof}

We now complete the proof of the Main Theorem.  Let $\mathcal{K}$ be as in
Corollary~\ref{cor.finite}.  In view of Corollary~\ref{cor.finite}, it suffices
to show that for each $K\in\mathcal{K}$, each metric in $\mc_K$ is spectrally
isolated within $\mc_K$.  Let $g=g_{\alpha,h}\in\mc_K$.   Choose a neighborhood
$W$ of $g$ in $\mc_K$ as in Lemma~\ref{lem.al} so that for  $g_{\alpha',h'}\in
W$, isospectrality of $g_{\alpha',h'}$ to $g$ implies $\alpha=\alpha'$.   Choose
$\beta\in\R$ so that $h\in\Adm(K,\beta)$.  Let $\hh=F(h)$ as in
Proposition~\ref{prop.beta}.  By Lemma~\ref{lemma.rep}, for each $p=1,\dots,s$,
there exists $\zeta_p\in\Spec(G,\tg_{\alpha,\beta})$ so that
$\zeta_p+\gamma_p(\hh')\in\Spec(G,g_{\alpha,h'})$ for all $h'\in\Adm(K,\beta)$. 
Choose $\epsilon>0$ so that for all $p=1,\dots,s$, the interval
$(\zeta_p+\gamma_p(\hh)-\epsilon, \zeta_p+\gamma_p(\hh)+\epsilon)$ contains no
eigenvalues in $\Spec(G,g_{\alpha,h})$ other than $\zeta_p+\gamma_p(\hh)$. 
Choose a neighborhood $V$ of $\hh$ in $\mb(K)$ so that for all $p$, we have
$|\gamma_p(\hh')-\gamma_p(\hh)|<\epsilon$ whenever $\hh'\in V$.  For $\hh'\in V$
distinct from $\hh$, we have $\gamma_p(\hh')-\gamma_p(\hh)\neq 0$ for at least
one choice of $p$, since the $\gamma_p$'s are a complete set of invariants. 
Thus letting $U=F^{-1}(V)$, we have $\Spec(G,g_{\alpha,h})\neq 
\Spec(G,g_{\alpha,h'})$ for all $h'\in U$.     Letting $\mathcal{O}=W\cap
\{g_{\alpha',h'}: \alpha'\in\R, h'\in U\}$, then $\mathcal{O}$ is an open
neighborhood  of $g$ in $\mc_K$ all of whose elements other than $g$ itself
have spectrum different from that of $g$.  This completes the proof.

\begin{rem}
Although we have focused on left-invariant naturally reductive metrics on a
compact simple Lie group, similar arguments demonstrate that for an arbitrary
compact Lie group $G$ metrics of the form given by Theorem~\ref{thm:NatRed} are
spectrally isolated within this class (see Remark~\ref{rem.natred}
(\ref{rem.GenNatRed})).
\end{rem}

%%%%%%%%%%%%%%%%%%%%%%%%%%%%%%%%%%%%%%%%%%%%%
%%%%%%%  The Spectra of Semi-Simple Naturally Reductive Metrics  %%%%%%%
%%%%%%%%%%%%%%%%%%%%%%%%%%%%%%%%%%%%%%%%%%%%%

\section{The Spectra of Semisimple Naturally Reductive
Metrics}\label{Sec:SemiNatRed}

It is rare that one is able to explicitly compute the spectra of Riemannian
manifolds. In this section we outline the computation of the spectrum of a class
of 
left-invariant naturally reductive metrics on compact simple Lie groups.   The
class of metrics we consider are the ``semisimply naturally reductive metrics''
that  we now define.

\begin{dfn}
A left-invariant naturally reductive metric $g$ on a compact Lie group $G$ will
be said to be {\bf semisimply naturally reductive}  if $g$ is naturally
reductive  with respect to $G \times K$ for some semisimple subgroup  $K \leq
G$. 
\end{dfn}

\begin{rem}\label{rem.h}
In the notation of Theorem~\ref{thm:NatRed}, suppose $G$ is a compact simple Lie
group and the metric $g = g_{\alpha, h}=e^{\alpha} g_{0} \upharpoonright
\germ{p} \oplus h \upharpoonright \germ{K}$ is naturally reductive  with respect
to $G\times K$ for some semisimple subgroup $K \leq G$.  Since the metric $h$ on
$K$ is bi-invariant, we necessarily have 
$$h =  e^{\alpha_1} g_0 \upharpoonright \germ{K}_1 \oplus \cdots \oplus
e^{\alpha_r} g_0 \upharpoonright \germ{K}_r,$$ 
where $\germ{K}$, the Lie algebra of $K$, is the direct sum of simple factors 
$\germ{K}_1, \ldots , \germ{K}_r$.

If some $\alpha_{i} = \alpha$, then the factor $\germ{K}_{i}$ can be absorbed
into $\germ{p}$.  Thus we will always assume that $\alpha_i\neq \alpha$ for all
$i=1,\dots,r$.

\end{rem}

The next result shows that for a simple Lie group $G$ every left-invariant
semisimply naturally reductive metric can be
realized as the base space of a semi-Riemannian submersion (with totally
geodesic fibers), where the total space is a compact Lie group equipped with a
\emph{bi-invariant} semi-Riemannian metric. This coupled with
Theorem~\ref{Thm:Casimir} below will allow us to explicitly compute
the spectra of left-invariant semisimply naturally reductive metrics on simple
Lie groups.  Recall that a metric $g$ on a manifold $M$ is said to be
semi-Riemannian if for each $p \in M$ we have that $g_p : T_{p}M \times T_{p} M
\to
\R$ is symmetric and non-degenrate.

\begin{prop}\label{prop:SemiRiem}
Let $G$ be a compact simple Lie group, let $K \leq G$ be a closed semisimple
subgroup, let $g=g_{\alpha,h}$ be a left-invariant metric on $G$ naturally
reductive with respect to $G\times K$, and let $\alpha, \alpha_1,\alpha_r$ be as
in Remark \ref{rem.h}.   Let $\hh$ be
the bi-invariant semi-Riemannian metric on $K$ given by $$\hh = \beta_{1}
g_{0} \upharpoonright \germ{K}_{1} \oplus \cdots \oplus \beta_{r}  g_{0}
\upharpoonright \germ{K}_{r} \;\;(
\beta_{i} = \frac{e^{\alpha_i }}{1-e^{(\alpha_i -\alpha)}}  =
\frac{e^{\alpha}}{e^{(\alpha - \alpha_i)} -1 }) .$$
Then the map $$\Phi: (G \times K, e^{\alpha} g_{0} \times \hh ) \to
(G,g_{\alpha,h})$$ given by 
$$(g,k) \mapsto gk^{-1}$$
is a semi-Riemannian submersion with totally geodesic fibers. 
\end{prop}

\begin{rems}
The submersion $\Phi$ will be Riemannian if $ \alpha_{i}< \alpha$ for each
$i$. Otherwise, it will be only semi-Riemannian.

In Corollary~\ref{cor.beta} we realized every naturally reductive metric
$g_{\alpha,h}$ on a simple Lie group $G$ as the base space of a Riemannian
submersion $\Phi: (G \times K, \tilde{g} \times \tilde{h}) \to (G,g)$, where the
metric $  \tilde{g} \times \tilde{h}$ is not necessarily bi-invariant. To obtain
a bi-invariant metric on the total space, as in Proposition~\ref{prop:SemiRiem},
we need to allow semi-Riemannian metrics. 
\end{rems}

\begin{proof}
Write $h(X,Y) = g_0(A_h X, Y)$, where $A_h: \germ{K} \to \germ{K}$ is a
$g_0$-self-adjoint linear transformation (with non-zero eigenvalues).  Following
the notation of \ref{nota.h} we let $\mbsemi(K)$ denote the space of
bi-invariant semi-Riemannian metrics on $K$. Since $e^{\alpha}$ is not an
eigenvalue of $A_h$, we may mimic the argument and notation of
Proposition~\ref{prop.beta} and Corollary~\ref{cor.beta} to see that there is a
\emph{unique} $\hh \in \mbsemi(K)$ such that: 
\begin{enumerate}
\item The semi-Riemannian metric $e^{\alpha}g_0 \times \hh$ on $K \times K$ is
non-degenerate on $\Delta \germ{K} \leq \germ{K} \times \germ{K}$. 
\item The map $\Phi: (K \times K, e^{\alpha}g_0 \times \hh) \to (K, h)$ is a
semi-Riemannian submersion.
\item The semi-Riemannian metric $e^\alpha g_0 \times \hh$  on $G \times K$ is
non-degenerate when restricted to $\Delta \germ{K} \leq \germ{g} \times
\germ{K}$ and $\Phi: (G \times K, e^\alpha g_0 \times \hh) \to (G,
g_{\alpha,h})$ is a semi-Riemannian submersion with totally geodesic fibers.
\end{enumerate}
\end{proof}
 
It is well known that the Laplace operator on a compact simple Lie group
equipped with the bi-invariant metric is given by the Casimir operator.
Consequently, in this case, we may explicitly compute the spectrum of the
Laplacian in terms of the highest weights of irreducible representations.   For
convenience, given a compact Lie group $G$ we will identify  $\widehat{G}$, the
space of irreducible unitary representations of $G$ (see \ref{rept}), with the
collection of highest weights of irreducible representations of $G$ with respect
to some choice of positive Weyl chamber.

\begin{thm}[\cite{Fegan}]\label{Thm:Casimir}
Let $K$ be a simply-connected semisimple compact Lie group and let $B$ denote
the Killing form on its Lie
algebra $\germ{K}$. We let $g_0$ denote the bi-invariant metric on $K$
determined by $-B$ and let $\| \cdot \|$ be the norm on $\germ{K}^*$, the dual
space to $\germ{K}$, determined by $-B$. Then 
$$\Spec(K, g_o) = \{ c(\lambda) : \lambda \in \widehat{G} \}$$
where $c(\lambda) = \| \lambda + \rho\|^2 - \| \rho \|^2$ and $\rho$ is half the
sum of the positive roots. 
\end{thm}

The preceding result allows us to compute the spectrum of any semisimple Lie
group endowed with a bi-invariant metric. Indeed, let $K$ be a connected compact
semisimple Lie group with Lie algebra $\germ{K}= \germ{K}_1 \oplus \cdots \oplus
\germ{K}_r$, where $\germ{K}_i$ is simple for each $i = 1 , \ldots , r$.  We
caution that $K$ is not necessarily a direct product of simple factors; the
simple connected normal subgroups may have non-trivial, but necessarily finite,
intersections.  Let $\widetilde{K} = \widetilde{K}_1 \times \cdots \times
\widetilde{K}_r$ be the universal cover of $K$ with covering map $\pi$, here
$\widetilde{K}_i$ is the unique simply-connected Lie group with Lie algebra
$\germ{K}_i$. Then $K = \widetilde{K} / \Gamma$ for some finite subgroup $\Gamma
\leq Z(\widetilde{K})$ of the center of $\widetilde{K}$. It follows that
$\widehat{K}$, the set of equivalence classes of irreducible unitary
representations of $K$, can be identified with 
$$\widehat{K} \equiv \{ \sigma \in \widehat{\widetilde{K}} : \Gamma \leq \ker
(\sigma) \}.$$
Since, moreover,  $\widehat{\widetilde{K}}$ can be identified with
$\widehat{\widetilde{K}}_1 \times \cdots \times \widehat{\widetilde{K}}_r$ (see
\ref{rept} (2)), we may consider $\widehat{K}$ to be a subset of 
$\widehat{\widetilde{K}}_1 \times \cdots \times \widehat{\widetilde{K}}_r$. Now,
let $g = e^{\alpha_1}g_0 \upharpoonright \germ{K}_1 \oplus \cdots \oplus
e^{\alpha_r}g_0 \upharpoonright \germ{K}_r$ be a bi-invariant metric on $K$ and
let $\tilde{g}$  (given by the same expression as $g$)  be the bi-invariant
metric on $\widetilde{K}$ such that $\pi: (\widetilde{K}, \tilde{g}) \to (K, g)$
is a Riemannian covering. Then it follows from Theorem~\ref{Thm:Casimir} that 
$$\Spec(K, g) = \{ \sum_{1}^{r} e^{-\alpha_i} c_i(\lambda_i) : (\lambda_1,
\cdots , \lambda_r) \in \widehat{K} \},$$
where for $\lambda_i\in \widehat{\widetilde{K}}_i$, $c_i(\lambda_i)$ is the
expression defined in Theorem \ref{Thm:Casimir} with $K$ replaced by
$\widetilde{K}_i$.

We now compute the spectrum of an arbitrary semisimply naturally reductive
metric $g$ on a compact connected simple Lie group $G$.  We express $g$ as in
Remark \ref{rem.h}.   Let  $(G \times K, e^{\alpha} g_{0} \times
 \hh)$ be the associated semi-Riemannian manifold given in
Proposition~\ref{prop:SemiRiem}.   As in the Riemannian case, one defines the
Laplacian $\Delta$ on a semi-Riemannian manifold as $\Delta f =- \Div (\Grad
f)$.
In the semi-Riemannian setting the eigenvalues of the Laplace operator can be
both
positive and negative. It follows from Propostion \ref{prop:SemiRiem}, in
particular the bi-invariance of the metric $e^{\alpha} g_{0} \times \hh$, and
the discussion above  that the
spectrum of $(G \times K, e^{\alpha} g_{0} \times \hh)$ is given by the
collection of numbers
$$e^{-\alpha} \{ \| \lambda + \rho \|^{2} - \|\rho \|^{2} + \sum_{i = 1}^{r} 
(e^{ \alpha - \alpha_{i}} - 1) (\| \lambda_i + \rho_i \|^{2}_{i} - \| \rho_i
\|^{2}_{i})\},$$
as $\lambda$ varies over $\widehat{G}$ and $(\lambda_{1}, \cdots , \lambda_r)$
varies over $\widehat{K} \subset \widehat{\widetilde{K}}_1 \times \cdots \times
\widehat{\widetilde{K}}_r$. Since $\Phi : (G
\times K, e^{\alpha} g_{0} \times \hh) \to (G,g)$ is a
semi-Riemannian submersion with totally geodesic fibers (isometric to $\Delta
K$) it follows that the Laplacian of $(G, g)$ is given by the restriction of the
Laplacian on $G \times K$ to the $\Delta K$-invariant functions. Hence, the
discussion in \ref{rept} (2) implies that if we let 
$\widehat{G\times K}^{*} = \{ \sigma \otimes \tau \in \widehat{G\times K} :
[\sigma \upharpoonright K : \overline{\tau}] \neq 0  \} \subset \widehat{G}
\times \widehat{K}$, then  
$$\Spec(G,g_{\alpha, h}) = \{  e^{-\alpha} ( c(\lambda) + \sum_{i = 1}^{r}  (e^{
\alpha - \alpha_{i}} - 1)c_{i}(\lambda_{i}) ) : (\lambda; \lambda_1, \ldots ,
\lambda_r) \in  \widehat{G\times K}^{*} \}.$$

\begin{rem}
Although we have concentrated on compact simple groups, the computation above
can be generalized to obtain formulas for the known semisimply naturally
reductive metrics on an arbitrary compact Lie group $G$.
\end{rem}

%%%%%%%%%%%%%%%%%%%%%%%%%%%%
%%%%%%%%%% Bibliography  %%%%%%%%%%%
%%%%%%%%%%%%%%%%%%%%%%%%%%%%

%\bibliographystyle{amsplain}
\bibliographystyle{amsalpha}

\end{document}